\documentclass[11pt]{amsart}

\usepackage[T1]{fontenc}
\usepackage{lmodern}
\usepackage{mathtools,amssymb}
\usepackage{microtype}
\usepackage[hidelinks]{hyperref}

\allowdisplaybreaks
\numberwithin{equation}{section}
\numberwithin{figure}{section}
\numberwithin{table}{section}
\newtheorem{theorem}{Theorem}[section]
\newtheorem{proposition}[theorem]{Proposition}
\newtheorem{lemma}[theorem]{Lemma}
\newtheorem{corollary}[theorem]{Corollary}
\theoremstyle{definition}

\theoremstyle{remark}

\newtheorem*{remarkx}{Remark}

\newcommand{\R}{\mathbb R}

\title{Mixed Radial Volume Comparison under Spectral Ricci Bounds}
\author{Han Hong}
\address{Department of Mathematics and Statistics \\
Beijing Jiaotong University \\
Beijing, China 100044}
\email{hanhong@bjtu.edu.cn}

\author{Gaoming Wang}
\address{Beijing Institute of Mathematical Sciences and Applications \\
Beijing, China 100044}
\email{gaomingwang@bimsa.cn}
\date{}

\begin{document}

\begin{abstract}
Let $(M^n,g)$ be complete, let $u>0$, and assume
\[
 \operatorname{Ric}_g-\alpha\frac{\Delta_g u}{u}g
 \ge (n-1)\kappa g.
\]
We introduce mixed radial balls associated with the conformal metric
$u^{2\alpha}g$.
For these mixed radial balls, we obtain a space-form-sharp model comparison and
polynomial weighted volume growth without pointwise bounds on $u$.
As applications, we give a radial derivation  of the spectral Bonnet--Myers
and sharp volume theorem of Antonelli--Xu, and a short volume growth
derivation of the stable Bernstein theorem in $\R^4$.
\end{abstract}

\maketitle

\section{Introduction}

Let $(M^n,g)$ be a complete Riemannian manifold with $n\ge3$, let $u$ be a
smooth positive function on $M$, and suppose that, for some $\alpha>0$ and
$\kappa\in\mathbb R$,
\begin{equation}\label{eq:spectral}
 \operatorname{Ric}_g-\alpha\frac{\Delta_g u}{u}g
 \ge (n-1)\kappa g.
\end{equation}
Spectral Ricci inequalities of this type arise naturally from stability and
the Gauss equations for minimal hypersurfaces under bi-Ricci curvature
bounds; see Shen and Ye~\cite{ShenYe1996}.
They were later studied as curvature assumptions in their own right, leading
to further geometric estimates; see Carron and Rose~\cite{CarronRose2021}.
Recent progress includes sharp global volume and diameter comparison
\cite{ChodoshLiMinterStryker2026,AntonelliXu2024}, sharp splitting and
criticality theorems with applications to the topology of stable minimal
hypersurfaces
\cite{AntonelliPozzettaXu2024,CatinoMariMastroliaRoncoroni2024,
HongWangBoundary2026}.
Other recent developments concern volume comparison with boundary,
band-width estimates, and rigidity or flexibility
\cite{JiaLi2025,ChaiSunBand2026,AntonelliXuConnectedSum2026,ChaiSun2026,
AntonelliLiSweeney2026,HongWangEnds2026}.
Related spectral comparison results have also been established for
Bakry--\'Emery Ricci tensors \cite{ChuHao2026,Wu2026}.
In our previous joint work~\cite{HongWangBoundary2026}, the second variation
of weighted geodesics was used to study geometric properties of manifolds
under spectral Ricci bounds.
Motivated by that approach, in the present paper we develop comparison
geometry in a weighted radial setting.
In particular, we derive comparison estimates for a weighted radial Jacobian
and use them to study weighted volume growth on mixed radial domains.

This weighted radial comparison is organized around the conformal metric
$h_\alpha=u^{2\alpha}g$, which is adapted to the scalar spectral term in
\eqref{eq:spectral}.
Along its radial geodesics, classical geodesic-polar comparison is obstructed
by the derivatives of $u$ that enter the radial Riccati equation.
We compensate for these terms by weighting both the radial Jacobian and the
radial parameter with suitable powers of $u$.
The resulting weighted Jacobian satisfies an exact Riccati formula.
Two different estimates derived from this formula produce complementary
volume growth conclusions for mixed radial domains, while keeping the
positive curvature term produces an independent radial
cutoff.

We begin by defining the sets on which the comparison is carried out.
Fix $o\in M$ and let $\mathbb S_oM$ be the $g$-unit sphere in $T_oM$.
For $\theta\in\mathbb S_oM$, let $\gamma_\theta$ be the image of the
$h_\alpha$-geodesic from $o$ in direction $\theta$, parametrized by
$g$-arclength $s$.
Let $c_\alpha(\theta)$ denote its first $h_\alpha$-cut time in the
$s$-parameter, truncated also at the maximal existence time when
$h_\alpha$ is incomplete.
For $p\in\mathbb R$, set
$\rho_p(s,\theta):=\int_0^s u^p(\gamma_\theta(t))\,dt$, and define the mixed
radial ball
\begin{equation*}
 \Omega_R^{\alpha,p}(o)
 :=\left\{
 \gamma_\theta(s):
 \theta\in\mathbb S_oM,\quad
 0\le s<c_\alpha(\theta),\quad
 \rho_p(s,\theta)<R
 \right\}.
\end{equation*}
The choice $p=0$ gives the $g$-arclength mixed radial ball
\begin{equation*}
 \Omega_R^\alpha(o)
 :=\Omega_R^{\alpha,0}(o)
 =\left\{
 \gamma_\theta(s):
 \theta\in\mathbb S_oM,\quad
 0\le s<\min\{R,c_\alpha(\theta)\}
 \right\}.
\end{equation*}
If $p=\alpha$ and $h_\alpha$ is complete, then
$\Omega_R^{\alpha,\alpha}(o)$ agrees with the genuine metric ball
$B_R^{h_\alpha}(o)$ up to the cut locus.
Thus the notation $\Omega$ distinguishes the general radial construction
from an ordinary metric ball.

For $\kappa\in\mathbb R$, let $s_\kappa$ solve
$s_\kappa''+\kappa s_\kappa=0$, with $s_\kappa(0)=0$ and
$s_\kappa'(0)=1$.
If $s_\kappa$ has a positive zero, let $\overline s_\kappa$ be its zero
extension after the first one; otherwise set
$\overline s_\kappa=s_\kappa$.
Write $\omega_{n-1}=|\mathbb S^{n-1}|$.
Set $\eta:=(n-3)\alpha/(n-1)$, and denote the model volume by
$V_\kappa^n(R):=\omega_{n-1}\int_0^R\overline
s_\kappa^{\,n-1}(t)\,dt$.

\subsection{Main results}

The first result is a scale-invariant comparison for the $g$-arclength
member of the mixed radial family.

\begin{theorem}\label{thm:model-growth}
Assume $\kappa\ge0$ and $0<\alpha\le(n-1)/(n-2)$.
For $R>0$, let $m_R:=\inf_{\Omega_R^\alpha(o)} u$.
Then $m_R>0$ and
\begin{equation}\label{eq:model-growth}
 \int_{\Omega_R^\alpha(o)}
 \left(\frac{u}{u(o)}\right)^{2\alpha/(n-1)}dV_g
 \le
 \left(\frac{u(o)}{m_R}\right)^{n\eta}V_\kappa^n(R)
 \qquad(R>0).
\end{equation}
If $o$ is a global minimum point of $u$, then
\begin{equation}\label{eq:minimum-growth}
 \operatorname{Vol}_g\bigl(\Omega_R^\alpha(o)\bigr)
 \le V_\kappa^n(R).
\end{equation}
\end{theorem}

The second result establishes polynomial weighted-volume growth for mixed
radial balls in an explicit parameter range.
Call a triple $(\alpha,p,\sigma)$ admissible if
$0<\alpha<(n-1)/(n-2)$ and
\begin{align}
 \bigl((n-1)p+(n-3)\alpha\bigr)^2
 &<4\alpha\bigl((n-1)-(n-2)\alpha\bigr),
 \notag\\
 \bigl((n-1)\sigma-2\alpha\bigr)^2
 &<4n\alpha\bigl((n-1)-(n-2)\alpha\bigr).
 \label{eq:admissible-sigma}
\end{align}
The associated explicit exponent $K(\alpha,p,\sigma)$ is given in
\eqref{eq:K} below.

\begin{theorem}\label{thm:polynomial}
Assume $\kappa\ge0$.
For every admissible triple $(\alpha,p,\sigma)$ there is a constant
$C=C(n,o,u,\alpha,p,\sigma)>0$ such that
\begin{equation}\label{eq:polynomial-growth}
 \int_{\Omega_R^{\alpha,p}(o)}u^\sigma\,dV_g
 \le C\left(1+R^{K(\alpha,p,\sigma)}\right)
 \qquad(R>0).
\end{equation}
\end{theorem}

\begin{remarkx}
The exponent can attain the Euclidean value $n$.
More precisely, $K(\alpha,p,\sigma)\ge n$ for every admissible triple, and
equality holds exactly when $\sigma=np+(n-2)\alpha$.
\end{remarkx}

The third result is independent of the polynomial-growth admissibility
conditions and concerns the positive-curvature geometry of $g$-arclength
mixed radial balls.

\begin{theorem}\label{thm:positive-diameter}
If $\kappa>0$ and $0<\alpha<4/(n-1)$, then
the conformal metric $h_\alpha$ is complete and
\begin{equation}\label{eq:diameter-free}
 \operatorname{diam}_g(M)
 \le
 2\pi
 \sqrt{
 \frac{(n-1)-(n-2)\alpha}
 {(n-1)\bigl(4-(n-1)\alpha\bigr)\kappa}
 }.
\end{equation}
The constant in \eqref{eq:diameter-free} is optimal among smooth closed
examples.
\end{theorem}

After matching parameters, the constant in \eqref{eq:diameter-free} is the
one obtained by Shen and Ye~\cite[Corollary~1]{ShenYe1997}.
The argument below gives a radial derivation of this estimate and proves its
optimality among smooth closed examples.
Antonelli and Xu~\cite[Corollary~1]{AntonelliXu2024} proved that, in this
parameter range, $\operatorname{diam}_g(M)\le C(n,\alpha)\kappa^{-1/2}$
for some constant $C(n,\alpha)$.
For $n=3$, equality is attained by a round sphere and a constant function.
For $n>3$, the proof below constructs smooth rotationally symmetric spheres
whose diameters converge to the right-hand side of
\eqref{eq:diameter-free}.

\subsection{Applications}

We use the two mixed radial volume mechanisms to derive two geometric
applications.

\begin{corollary}\label{cor:intro-AX}
Let $(M^n,g)$ be closed, $n\ge3$, let $\kappa>0$, and let
$0<\alpha\le(n-1)/(n-2)$.
Suppose that a positive function $u$ satisfies \eqref{eq:spectral}.
Let $(\widetilde M,\widetilde g)$ be the universal cover and let
$\widetilde u$ be the lift of $u$.
Then
\begin{equation}\label{eq:AX-diameter}
 \operatorname{diam}_{\widetilde g}(\widetilde M)
 \le
 \frac{\pi}{\sqrt\kappa}
 \left(\frac{\max_M u}{\min_M u}\right)^\eta,
 \qquad
 \eta=\frac{n-3}{n-1}\alpha,
\end{equation}
so $\pi_1(M)$ is finite, and
\begin{equation}\label{eq:AX-volume}
 \operatorname{Vol}_{\widetilde g}(\widetilde M)
 \le
 \kappa^{-n/2}\operatorname{Vol}(\mathbb S^n).
\end{equation}
If equality holds in \eqref{eq:AX-volume}, then $u$ is constant and
$(\widetilde M,\widetilde g)$ is the round sphere of radius
$\kappa^{-1/2}$.
\end{corollary}

This recovers the sharp spectral Bonnet--Myers and volume theorem of
Antonelli and Xu~\cite{AntonelliXu2024} by a radial comparison argument.
Compared with their global result, Theorem~\ref{thm:model-growth} gives a
local version of volume comparison on the mixed radial domains
$\Omega_R^\alpha(o)$.

The second application concerns the stable Bernstein problem.
The classical theorem of Schoen, Simon, and Yau gives flatness under an
additional Euclidean extrinsic volume growth assumption
\cite{SchoenSimonYau1975}; Bellettini recently extended the corresponding
immersed result through hypersurface dimension six
\cite{Bellettini2025}.
Without a growth assumption, the case $M^3\to\R^4$ was first settled by
Chodosh and Li~\cite{ChodoshLi2024}; subsequent proofs and extensions use
anisotropic and warped $\mu$-bubble methods, conformal Bakry--\'Emery
geometry, and spectral Green-kernel estimates
\cite{ChodoshLiAnisotropic2023,CatinoMastroliaRoncoroni2024,
CabreCatinoMariMastroliaRoncoroni2026}.
The cases $M^4\to\R^5$ and $M^5\to\R^6$ were later resolved by Chodosh,
Li, Minter, and Stryker~\cite{ChodoshLiMinterStryker2026} and by
Mazet~\cite{Mazet2024}, respectively; Mazet's argument uses the
four-dimensional spectral volume comparison of Antonelli and Xu.
The general problem for $M^6\to\R^7$ remains open without an additional
growth hypothesis.
Our proof below combines the polynomial volume estimate proved here with the
weighted $L^4$ curvature inequality of Cabr\'e, Catino, Mari, Mastrolia, and
Roncoroni~\cite{CabreCatinoMariMastroliaRoncoroni2026}.

\begin{corollary}\label{cor:bernstein}
Let $F:(M^3,g)\longrightarrow\R^4$ be a smooth, connected, complete,
two-sided, stable minimal immersion.
Then $A\equiv0$, and $F(M)$ is an affine three-plane in $\R^4$.
\end{corollary}

The two volume estimates begin with the same one-dimensional quantity: a
weighted Jacobian along each radial $h_\alpha$-geodesic.
The resulting Riccati formula can be used in two ways.
Taking the $(n-1)$-st root of the weighted Jacobian converts the formula into
a linear second-order inequality and yields the space-form comparison in
Theorem~\ref{thm:model-growth}.
For general mixed radial balls, we instead retain the logarithmic derivatives
of both the weighted Jacobian and $u$ as coupled variables.
A quadratic estimate for these variables, followed by a correction involving
the Jacobian derivative, gives the monotonicity formula from which
Theorem~\ref{thm:polynomial} follows.

Section~\ref{sec:riccati} derives the weighted radial formula.
Section~\ref{sec:model} develops the $g$-arclength mixed radial comparison,
including model comparison, positive-curvature cutoff, and optimality of the
diameter constant.
Section~\ref{sec:polynomial} proves polynomial growth for general mixed
radial balls.
Sections~\ref{sec:AX} and~\ref{sec:bernstein} prove the two applications.

\textit{Use of AI.}
Artificial-intelligence tools (ChatGPT 5.5 and 5.6) were used to assist with some of the
calculations in the preparation of this manuscript.  The authors conceived the research direction and developed the weighted-geodesic approach and the framework of mixed radial balls. All AI-assisted calculations were
independently checked by the authors, who take full responsibility for the
arguments and conclusions of the paper.

\section{Weighted radial Jacobians under conformal change}\label{sec:riccati}

Throughout the technical sections, assume $n\ge3$ and \eqref{eq:spectral}.
We begin with the radial calculation used in all three conclusions.
Put $\phi:=\alpha\log u$ and $h:=e^{2\phi}g=u^{2\alpha}g$.
At a regular point of the $h$-distance $r=d_h(o,\cdot)$, let
$E:=\nabla^h r$ and $T:=u^\alpha E$.
Then $|E|_h=|T|_g=1$.
Along a radial ray parametrized by $g$-arclength $s$, we have
$dr/ds=u^\alpha$ and $\partial_s=u^\alpha\partial_r$.

Let $\widetilde J(r,\theta)$ be the ordinary $h$-polar Jacobian and let
$J(s,\theta)$ be the $g$-Jacobian of the corresponding radial family.
Since $dV_h=u^{n\alpha}dV_g$ and $dr=u^\alpha ds$, we have
$\widetilde J=u^{(n-1)\alpha}J$.
The weighted density adapted to \eqref{eq:spectral} is
\begin{equation}\label{eq:j}
 j:=u^{-(n-2)\alpha}\widetilde J=u^\alpha J.
\end{equation}
Let $m:=\partial_s\log j$ and $\xi:=\phi_s$.
If $\mathcal S:=\operatorname{Hess}_h r|_{E^\perp}$ and
$H:=\operatorname{tr}_h\mathcal S$, write $\mathring{\mathcal S}$ for the
trace-free part of $\mathcal S$.
Finally, define the nonnegative radial defect
$\mathcal E_\alpha(T):=\operatorname{Ric}_g(T,T)
-\alpha(\Delta_gu/u)-(n-1)\kappa$.

The weighted radial quantities satisfy the following exact Riccati formula.

\begin{proposition}\label{prop:riccati}
Before the $h$-cut locus,
\begin{align}
 &m_s+\frac1{n-1}m^2+\frac{n-3}{n-1}m\xi
 +\left(\frac1\alpha-\frac{n-2}{n-1}\right)\xi^2+(n-1)\kappa
 \notag\\
 &\qquad
 =-u^{2\alpha}|\mathring{\mathcal S}|_h^2
  -\mathcal E_\alpha(T)
  -\frac1\alpha|\nabla_T^\perp\phi|_g^2.
 \label{eq:exact-riccati}
\end{align}
In particular, if $\kappa\ge0$ and $0<\alpha<(n-1)/(n-2)$, then
\begin{equation}\label{eq:scalar-riccati}
 m_s+\frac1{n-1}m^2+\frac{n-3}{n-1}m\xi
 +\left(\frac1\alpha-\frac{n-2}{n-1}\right)\xi^2
 \le0.
\end{equation}
\end{proposition}

\begin{proof}
The conformal transformation formulas give
\begin{align*}
 \operatorname{Ric}_h
 &=\operatorname{Ric}_g-(n-2)\operatorname{Hess}_g\phi
 +(n-2)d\phi\otimes d\phi
 \notag\\
 &\quad
 -\left(\Delta_g\phi+(n-2)|\nabla\phi|_g^2\right)g,
 \\
 \operatorname{Hess}_h\phi
 &=\operatorname{Hess}_g\phi-2d\phi\otimes d\phi
 +|\nabla\phi|_g^2g.
\end{align*}
Consequently,
\begin{equation}\label{eq:conformal-combination}
 \operatorname{Ric}_h+(n-2)\operatorname{Hess}_h\phi
 =\operatorname{Ric}_g-(\Delta_g\phi)g
 -(n-2)d\phi\otimes d\phi.
\end{equation}
Since $\phi=\alpha\log u$, we have
$\Delta_g\phi=\alpha(\Delta_gu/u)-|\nabla\phi|_g^2/\alpha$.
Evaluating \eqref{eq:conformal-combination} on
$E=u^{-\alpha}T$, multiplying by $u^{2\alpha}$, and splitting the gradient
of $\phi$ into its $T$ and $T^\perp$ parts gives
\begin{align}
 &u^{2\alpha}
 \left(\operatorname{Ric}_h
 +(n-2)\operatorname{Hess}_h\phi\right)(E,E)
 \notag\\
 &\qquad
 =(n-1)\kappa+\mathcal E_\alpha(T)
 +\left(\frac1\alpha-(n-2)\right)\xi^2
 +\frac1\alpha|\nabla_T^\perp\phi|_g^2.
 \label{eq:radial-conformal-tensor}
\end{align}

Equation \eqref{eq:j} implies
\begin{equation}\label{eq:log-j-r}
 \partial_r\log j=H-(n-2)\phi_r.
\end{equation}
The traced radial Riccati equation is
\begin{equation}\label{eq:ordinary-riccati}
 H_r=-|\mathcal S|_h^2-\operatorname{Ric}_h(E,E).
\end{equation}
Differentiating the right-hand side of \eqref{eq:log-j-r} with respect to
$r$ and using \eqref{eq:ordinary-riccati}, we obtain
\begin{equation*}
 \partial_r\left(H-(n-2)\phi_r\right)
 =-|\mathcal S|_h^2
 -\left(\operatorname{Ric}_h
 +(n-2)\operatorname{Hess}_h\phi\right)(E,E).
\end{equation*}
Moreover,
\begin{equation}\label{eq:m-H}
 m=\partial_s\log j
 =u^\alpha\left(H-(n-2)\phi_r\right),
 \qquad
 (u^\alpha)_s=u^\alpha\xi.
\end{equation}
Therefore
\begin{equation}\label{eq:m-first}
 m_s=\xi m-u^{2\alpha}|\mathcal S|_h^2
 -u^{2\alpha}
 \left(\operatorname{Ric}_h
 +(n-2)\operatorname{Hess}_h\phi\right)(E,E).
\end{equation}

The normal space to the radial direction has dimension $n-1$, so
\begin{equation}\label{eq:shape-decomposition}
 |\mathcal S|_h^2
 =|\mathring{\mathcal S}|_h^2+\frac1{n-1}H^2.
\end{equation}
Equations \eqref{eq:m-H} and $n-2=(n-1)-1$ also give
\begin{equation}\label{eq:H-m-xi}
 u^\alpha H=m+(n-2)\xi.
\end{equation}
Substituting \eqref{eq:radial-conformal-tensor},
\eqref{eq:shape-decomposition}, and \eqref{eq:H-m-xi} into
\eqref{eq:m-first} proves \eqref{eq:exact-riccati} after collecting the
mixed and square terms.
Under the hypotheses of the last assertion, every discarded term has the
required sign, and \eqref{eq:scalar-riccati} follows.
\end{proof}

Taking a root of the density turns the identity into a second-order
comparison inequality.

\begin{lemma}\label{lem:root-comparison}
Assume $0<\alpha\le(n-1)/(n-2)$ and put $F:=j^{1/(n-1)}$.
Along a regular radial ray, with $U(s)=u(\gamma(s))$,
\begin{equation}\label{eq:F-second-order}
 \left(U^\eta F'\right)'+\kappa U^\eta F\le0.
\end{equation}
If $y$ solves
\begin{equation}\label{eq:y}
 \left(U^\eta y'\right)'+\kappa U^\eta y=0,
 \qquad
 y(0)=0,
 \qquad
 y'(0)=1,
\end{equation}
and $\overline y$ is its zero extension after the first positive zero, then
before the cut time
\begin{equation}\label{eq:Jacobian-comparison}
 u^\alpha J=j=F^{n-1}
 \le u^\alpha(o)\overline y^{\,n-1}.
\end{equation}
\end{lemma}

\begin{proof}
Since $m=(n-1)F'/F$, the sum of the first two terms in
\eqref{eq:exact-riccati} is $(n-1)F''/F$.
Dropping the nonnegative square terms and using
$\eta=(n-3)\alpha/(n-1)$ gives \eqref{eq:F-second-order}.

While $F$ and $y$ are positive, the weighted Wronskian satisfies
\begin{equation*}
 \left[U^\eta(F'y-Fy')\right]'
 =y\left(U^\eta F'\right)'-F\left(U^\eta y'\right)'
 \le0.
\end{equation*}
The pole expansions are
\begin{equation*}
 \begin{aligned}
 J(s,\theta)&=s^{n-1}+O(s^n),\\
 F(s)&=u^{\alpha/(n-1)}(o)s+O(s^2),\\
 y(s)&=s+O(s^2).
 \end{aligned}
\end{equation*}
Thus the Wronskian has limiting value zero and $F/y$ is nonincreasing.
It follows that
\begin{equation}\label{eq:F-y}
 F\le u^{\alpha/(n-1)}(o)y
\end{equation}
until the first zero of $y$.
If this zero occurred before the cut time, letting $s$ approach it in
\eqref{eq:F-y} would contradict the positivity of $F$.
Extending both sides by zero after the relevant first zero gives
\eqref{eq:Jacobian-comparison}.
\end{proof}

\section{Mixed radial comparison in the
\texorpdfstring{$g$}{g}-arclength parameter}\label{sec:model}

\subsection{Sharp model comparison}

We first prove Theorem~\ref{thm:model-growth}.
The argument also explains why centering at a minimum point removes every
extraneous power of $u$.

\begin{proof}[Proof of Theorem~\ref{thm:model-growth}]
Every point of $\Omega_R^\alpha(o)$ is joined to $o$ by a curve of
$g$-length less than $R$.
Hence the radial ball lies in the compact set $\overline B_R^g(o)$, and
$m_R>0$.
Notice also that a radial $h_\alpha$-geodesic cannot cease to exist at finite
$g$-arclength while remaining in this compact set.
Thus no completeness assumption on $h_\alpha$ is needed here.

By \eqref{eq:Jacobian-comparison} and the identity
$2\alpha/(n-1)-\alpha=-\eta$,
\begin{equation}\label{eq:density-model-start}
 U^{2\alpha/(n-1)}J
 =U^{-\eta}j
 \le u^\alpha(o)U^{-\eta}\overline y^{\,n-1}.
\end{equation}
Introduce the increasing variable
\begin{equation}\label{eq:z}
 z(s):=\int_0^sU^{-\eta}(t)\,dt
\end{equation}
and write $Y(z)=y(s(z))$.
Equation \eqref{eq:y} becomes
\begin{equation}\label{eq:Y-equation}
 Y_{zz}+\kappa U^{2\eta}(s(z))Y=0,
 \qquad
 Y(0)=0,
 \qquad
 Y_z(0)=u^\eta(o).
\end{equation}
On the relevant part of the ray, $U\ge m_R$.
Since $\kappa\ge0$, Sturm comparison gives
$0\le\overline Y(z)\le u^\eta(o)\overline
s_{\kappa m_R^{2\eta}}(z)$.
Moreover, $s<R$ implies $z(s)\le R/m_R^\eta$.
Integrating \eqref{eq:density-model-start} first along a ray and then over
$\mathbb S_oM$, and using $dz=U^{-\eta}ds$, gives
\begin{align*}
 \int_{\Omega_R^\alpha(o)}u^{2\alpha/(n-1)}\,dV_g
 &\le
 \omega_{n-1}u^{\alpha+(n-1)\eta}(o)
 \int_0^{R/m_R^\eta}
 \overline s_{\kappa m_R^{2\eta}}^{\,n-1}(z)\,dz.
\end{align*}
The scaling identity
$s_{\kappa m_R^{2\eta}}(z)=m_R^{-\eta}s_\kappa(m_R^\eta z)$ and
$\alpha+(n-1)\eta=(n-2)\alpha$ give
\begin{equation*}
 \int_{\Omega_R^\alpha(o)}u^{2\alpha/(n-1)}\,dV_g
 \le u^{(n-2)\alpha}(o)m_R^{-n\eta}V_\kappa^n(R).
\end{equation*}
The exponent identity $(n-2)\alpha-2\alpha/(n-1)=n\eta$ shows that division
by $u^{2\alpha/(n-1)}(o)$ gives
\eqref{eq:model-growth}.
If $o$ is a global minimum point, then $m_R=u(o)$, and
\eqref{eq:minimum-growth} follows.
\end{proof}

The calculation also records a pointwise raywise refinement:
\begin{equation*}
 \int_{\Omega_R^\alpha(o)}u^{2\alpha/(n-1)}\,dV_g
 \le
 u^\alpha(o)
 \int_{\mathbb S_oM}\int_0^{\min\{R,c_\alpha(\theta)\}}
 U_\theta^{-\eta}(s)\overline y_\theta^{\,n-1}(s)
 \,ds\,d\theta.
\end{equation*}
The space-form estimate is obtained only after replacing the coefficient
$U^{2\eta}$ in \eqref{eq:Y-equation} by its lower bound.

\subsection{Positive-curvature cutoff and diameter estimate}

\begin{lemma}\label{lem:radial-cutoff}
If $\kappa>0$ and $0<\alpha<4/(n-1)$, then
\begin{equation}\label{eq:radial-cut-time}
 c_\alpha(\theta)
 \le
 2\pi
 \sqrt{
 \frac{(n-1)-(n-2)\alpha}
 {(n-1)\bigl(4-(n-1)\alpha\bigr)\kappa}
 }
 \qquad(o\in M,\ \theta\in\mathbb S_oM).
\end{equation}
\end{lemma}

\begin{proof}
Since $0<\alpha<4/(n-1)$, we have
$(n-1)-(n-2)\alpha>0$.
The Cauchy--Schwarz inequality gives
\begin{equation}\label{eq:positive-mixed-term}
 \frac{n-3}{n-1}m\xi
 \ge
 -\frac{\alpha(n-3)^2}
 {4(n-1)\bigl((n-1)-(n-2)\alpha\bigr)}m^2
 -\frac{(n-1)-(n-2)\alpha}{\alpha(n-1)}\xi^2.
\end{equation}
Combining \eqref{eq:positive-mixed-term} with
\eqref{eq:exact-riccati}, we obtain
\begin{equation}\label{eq:positive-riccati}
 m_s+a_\alpha m^2+(n-1)\kappa\le0,
\end{equation}
where
\begin{equation}\label{eq:positive-a-alpha}
 a_\alpha
 :=\frac1{n-1}
 -\frac{\alpha(n-3)^2}
 {4(n-1)\bigl((n-1)-(n-2)\alpha\bigr)}
 =\frac{4-(n-1)\alpha}
 {4\bigl((n-1)-(n-2)\alpha\bigr)}>0.
\end{equation}

Fix $0<\ell<c_\alpha(\theta)$ and let $\zeta$ be smooth on
$[0,\ell]$ with $\zeta(0)=\zeta(\ell)=0$.
Multiply \eqref{eq:positive-riccati} by $\zeta^2$ and integrate over
$[\varepsilon,\ell]$, where $0<\varepsilon<\ell$.
After integration by parts, the boundary term at $s=\ell$ vanishes because
$\zeta(\ell)=0$.
At the other endpoint, the pole expansion $m(s)=(n-1)/s+O(1)$ and
$\zeta(s)=O(s)$ imply $m(\varepsilon)\zeta^2(\varepsilon)=O(\varepsilon)$.
Letting $\varepsilon\downarrow0$ and applying the Cauchy--Schwarz inequality
gives
\begin{align}
 (n-1)\kappa\int_0^\ell\zeta^2\,ds
 &\le
 2\int_0^\ell m\zeta\zeta'\,ds
 -a_\alpha\int_0^\ell m^2\zeta^2\,ds
 \notag\\
 &\le
 \frac1{a_\alpha}\int_0^\ell(\zeta')^2\,ds.
 \label{eq:positive-index}
\end{align}
For $\zeta(s)=\sin(\pi s/\ell)$, we have
$\int_0^\ell(\zeta')^2\,ds=(\pi^2/\ell^2)\int_0^\ell\zeta^2\,ds$.
Substituting this identity into \eqref{eq:positive-index} gives
\begin{equation}\label{eq:ell-bound}
 \ell
 \le
 \frac{\pi}{\sqrt{a_\alpha(n-1)\kappa}}
 =
 2\pi
 \sqrt{
 \frac{(n-1)-(n-2)\alpha}
 {(n-1)\bigl(4-(n-1)\alpha\bigr)\kappa}
 }.
\end{equation}
Since \eqref{eq:ell-bound} holds for every
$0<\ell<c_\alpha(\theta)$, taking the supremum over $\ell$ proves
\eqref{eq:radial-cut-time}.
\end{proof}

The following standard escape argument makes the passage from the local
cutoff to completeness explicit.

\begin{lemma}\label{lem:minimizing-escape}
Let $g$ and $h$ be Riemannian metrics on $M$.
If $g$ is complete and $h$ is incomplete, then there are a point $o\in M$
and an $h$-geodesic $\gamma$, parametrized by $g$-arclength on $[0,\infty)$,
such that every compact subsegment of $\gamma$ is $h$-minimizing from $o$.
\end{lemma}

\begin{proof}
By the Hopf--Rinow theorem, the incompleteness of $h$ produces an
inextendible $h$-geodesic $\beta$ of finite $h$-length.
The local existence theorem for the geodesic equation shows that $\beta$
leaves every compact subset of $M$.
Set $o=\beta(0)$.
Since $g$ is complete, the closed balls $\overline B_R^g(o)$ are compact and
$\beta$ meets $\partial B_R^g(o)$ for all sufficiently large $R$.
The initial segment of $\beta$ up to its first contact with this boundary
has $h$-length at most $L_h(\beta)$.

For a sequence $R_j\to\infty$, minimize the $h$-length among curves in
$\overline B_{R_j}^g(o)$ that start at $o$ and end when they first reach
$\partial B_{R_j}^g(o)$.
Compactness of $\overline B_{R_j}^g(o)$ and the local equivalence of $g$ and
$h$ give a minimizing $h$-geodesic $\gamma_j$.
The first-contact property shows that every initial subsegment of $\gamma_j$
is $h$-minimizing in $M$.
Parametrize $\gamma_j$ by $h$-arclength.
Their lengths are bounded by $L_h(\beta)$, and their initial velocities lie
in the compact $h$-unit sphere at $o$.
After passing to a subsequence, their lengths converge to some
$L_\infty<\infty$ and their initial velocities converge to a vector $v$.
The corresponding $h$-geodesics converge on every compact interval on which
the geodesic $\gamma$ with initial velocity $v$ exists.

Every compact subsegment of $\gamma$ is $h$-minimizing from $o$, since it is
the limit of the corresponding subsegments of $\gamma_j$.
The maximal $h$-time of $\gamma$ is at most $L_\infty$: otherwise smooth
dependence on initial data would keep the endpoints of $\gamma_j$ in a fixed
compact set, whereas $\gamma_j$ reaches $\partial B_{R_j}^g(o)$ and
$R_j\to\infty$.
Thus $\gamma$ leaves every compact subset of $M$ in finite $h$-time.
Its $g$-length is infinite, and reparametrizing by $g$-arclength proves the
lemma.
\end{proof}

\begin{proof}[Proof of Theorem~\ref{thm:positive-diameter}]
Lemma~\ref{lem:radial-cutoff} proves \eqref{eq:radial-cut-time}.
Suppose for contradiction that $h_\alpha$ is incomplete.
Lemma~\ref{lem:minimizing-escape} gives an $h_\alpha$-geodesic parametrized
by $g$-arclength on $[0,\infty)$ whose every compact subsegment is
$h_\alpha$-minimizing.
Its $g$-arclength cut time is therefore infinite, contradicting
\eqref{eq:radial-cut-time}.
Thus $h_\alpha$ is complete.

For any $x,y\in M$, Hopf--Rinow provides a minimizing
$h_\alpha$-geodesic from $x$ to $y$.
If its $g$-length is $\ell$, then \eqref{eq:radial-cut-time}, applied with
base point $x$, gives
\begin{equation*}
 d_g(x,y)\le\ell
 \le
 2\pi
 \sqrt{
 \frac{(n-1)-(n-2)\alpha}
 {(n-1)\bigl(4-(n-1)\alpha\bigr)\kappa}
 }.
\end{equation*}
This proves \eqref{eq:diameter-free}.
The completeness of $g$ and Hopf--Rinow then imply that $M$ is compact.
The optimality assertion follows from
Proposition~\ref{prop:diameter-sharpness} below.
\end{proof}

\subsection{Optimality of the diameter constant}

The construction uses the following controlled smoothing lemma at the two
rotationally symmetric interfaces.

\begin{lemma}\label{lem:quantitative-smoothing}
Let $\delta>0$, and let $f,u:(-\delta,\delta)\to(0,\infty)$ be $C^1$
functions that are smooth on $(-\delta,0]$ and $[0,\delta)$.
Suppose that there is a constant $\mu>0$ such that, on
$(-\delta,0)\cup(0,\delta)$, the warped product data
\[
 g=ds^2+f^2(s)g_{\mathbb S^{n-1}},
 \qquad u=u(s),
\]
satisfy
\[
 \operatorname{Ric}_g
 -\alpha\frac{\Delta_gu}{u}g
 \ge\bigl((n-1)\kappa+\mu\bigr)g.
\]
Then, for every $0<\varepsilon<\delta$, there are smooth positive functions
$\widetilde f,\widetilde u:(-\delta,\delta)\to(0,\infty)$ that agree with
$f,u$ outside $(-\varepsilon,\varepsilon)$ and, for
$\widetilde g=ds^2+\widetilde f^2(s)g_{\mathbb S^{n-1}}$, satisfy
\[
 \operatorname{Ric}_{\widetilde g}
 -\alpha\frac{\Delta_{\widetilde g}\widetilde u}{\widetilde u}\widetilde g
 \ge(n-1)\kappa\widetilde g
\]
throughout $(-\delta,\delta)$.
\end{lemma}

\begin{proof}
The standard Ricci and radial Laplacian formulas for rotationally symmetric
warped products (see Petersen~\cite[Section~3.2.3, p.~71]{Petersen1998}) show
that the radial and unit tangential eigenvalues of the tensor on the left are
\begin{align}
 \mathcal R_s
 &=-(n-1)\frac{f''}{f}
 -\alpha\left(\frac{u''}{u}+(n-1)\frac{f'u'}{fu}\right),
 \label{eq:warped-radial}\\
 \mathcal R_\theta
 &=-\frac{f''}{f}
 +(n-2)\frac{1-(f')^2}{f^2}
 -\alpha\left(\frac{u''}{u}+(n-1)\frac{f'u'}{fu}\right).
 \label{eq:warped-tangential}
\end{align}
For fixed values of $f,f',u,u'$, both expressions are affine in
$f''$ and $u''$.
Hence the set of second derivative pairs for which both strict inequalities
hold is convex.

The hypothesis provides a uniform positive margin on both sides of the
interface.
Convolve the $C^1$ pair $(f,u)$ at scale $\varepsilon^2$.
Its second derivative is a convex average of the nearby one-sided second
derivatives, while the pair and its first derivative converge uniformly to
the original lower-order jets.
Thus the convolved pair satisfies the two strict inequalities on
$(-\varepsilon,\varepsilon)$ for small $\varepsilon$.

To keep the data unchanged outside the prescribed neighborhood, paste the
convolved pair back to the original pair with a cutoff varying on the larger
scale $\varepsilon$.
Because the original pair is $C^1$ and piecewise smooth, convolution at
scale $\varepsilon^2$ changes its zeroth and first jets by
$O(\varepsilon^4)$ and $O(\varepsilon^2)$, respectively.
The cutoff therefore contributes only $o(1)$ to the second derivatives.
The fixed strict margin absorbs this error.
The smoothed $u$ remains positive for sufficiently small $\varepsilon$, and
shrinking $\varepsilon$ places the smoothing inside any prescribed
neighborhood.
\end{proof}

\begin{proposition}\label{prop:diameter-sharpness}
Let $n\ge3$, $\kappa>0$, and $0<\alpha<4/(n-1)$.
There are smooth rotationally symmetric metrics $g_j$ on $\mathbb S^n$ and
smooth positive functions $u_j$ such that
\begin{equation}\label{eq:sharpness-bound}
 \operatorname{Ric}_{g_j}
 -\alpha\frac{\Delta_{g_j}u_j}{u_j}g_j
 \ge(n-1)\kappa g_j
\end{equation}
and
\begin{equation}\label{eq:sharpness-limit}
 \lim_{j\to\infty}\operatorname{diam}_{g_j}(\mathbb S^n)
 =
 2\pi
 \sqrt{
 \frac{(n-1)-(n-2)\alpha}
 {(n-1)\bigl(4-(n-1)\alpha\bigr)\kappa}
 }.
\end{equation}
\end{proposition}

\begin{proof}
Recall $a_\alpha$ from \eqref{eq:positive-a-alpha}, and set
$\beta_f:=2(2-\alpha)/(4-(n-1)\alpha)$ and
$\beta_u:=2(n-3)/(4-(n-1)\alpha)$.
For $\tau>0$, let $\kappa_\tau:=(1+\tau)\kappa$ and
$\lambda_\tau:=\sqrt{a_\alpha(n-1)\kappa_\tau}$,
and choose $0<c_\tau<1/(\sqrt{\beta_f}\,\lambda_\tau)$.
On $0<s<\pi/\lambda_\tau$, consider the singular rotationally symmetric
data
\begin{equation}\label{eq:singular-spindle}
 f_\tau(s):=c_\tau\sin^{\beta_f}(\lambda_\tau s),
 \qquad
 u_\tau(s):=\sin^{-\beta_u}(\lambda_\tau s).
\end{equation}
Substitution in \eqref{eq:warped-radial} gives
$\mathcal R_s=(n-1)\kappa_\tau$.
Moreover, \eqref{eq:warped-radial} and
\eqref{eq:warped-tangential} give the exact difference
$\mathcal R_\theta-\mathcal R_s=(n-2)(1-(f')^2+ff'')/f^2$,
where, for \eqref{eq:singular-spindle},
\begin{equation*}
 1-(f_\tau')^2+f_\tau f_\tau''
 =1-c_\tau^2\beta_f\lambda_\tau^2
 \sin^{2\beta_f-2}(\lambda_\tau s)>0.
\end{equation*}
Thus the singular body satisfies \eqref{eq:sharpness-bound} with a strict
margin.

When $n=3$, one has $\beta_f=1$, $\beta_u=0$, and
$\lambda_0=\sqrt\kappa$.
Taking $c_0=\kappa^{-1/2}$ gives the round sphere and a constant $u$, so
there is nothing to smooth.
Assume henceforth that $n>3$.
Then $\beta_f>1$ and $\beta_u>0$.

Fix a small $r>0$ and abbreviate $F_r:=f_\tau(r)$ and
$v_r:=f_\tau'(r)$.
For small $r$, $0<v_r<1$ and
\begin{equation*}
 F_r=c_\tau\lambda_\tau^{\beta_f}r^{\beta_f}
 +O(r^{\beta_f+2}),
 \qquad
 v_r=c_\tau\beta_f\lambda_\tau^{\beta_f}r^{\beta_f-1}
 +O(r^{\beta_f+1}).
\end{equation*}
Define
\begin{equation*}
 R_r:=\frac{F_r}{\sqrt{1-v_r^2}},
 \qquad
 \ell_r:=R_r\arccos v_r,
 \qquad
 f_{\mathrm{cap}}(t):=R_r\sin(t/R_r).
\end{equation*}
Then
\begin{equation}\label{eq:cap-matching}
 \begin{aligned}
 f_{\mathrm{cap}}(\ell_r)&=F_r,
\quad\quad f_{\mathrm{cap}}'(\ell_r)=v_r,\\
 \ell_r&=\frac\pi2c_\tau\lambda_\tau^{\beta_f}r^{\beta_f}
 +O\bigl(r^{2\beta_f-1}+r^{\beta_f+2}\bigr)
 =o(r).
 \end{aligned}
\end{equation}

Let $q_r$ be a smooth nonpositive function on $[0,\ell_r]$ that is linear
near zero, is nonincreasing, and satisfies $q_r(0)=0$,
$q_r(\ell_r)=(\log u_\tau)'(r)$, and $|q_r|\le C/r$.
Set
$u_{\mathrm{cap}}(t):=u_\tau(r)
\exp\left(-\int_t^{\ell_r}q_r(a)\,da\right)$.
The metric
$g_{\mathrm{cap}}:=dt^2+f_{\mathrm{cap}}^{\,2}(t)g_{\mathbb S^{n-1}}$ is a
geodesic cap in the round sphere of radius $R_r$.
Since $f_{\mathrm{cap}}(0)=0$, $f_{\mathrm{cap}}'(0)=1$, and
$f_{\mathrm{cap}}$ has a smooth odd extension across zero,
$g_{\mathrm{cap}}$ extends smoothly across the pole $t=0$.
Moreover, $q_r(0)=0$ and $q_r$ is linear near zero, so
$u_{\mathrm{cap}}$ has a smooth positive even extension across the pole.
At $t=\ell_r$, both functions and their first derivatives match the
singular body.

The cap metric satisfies
$\operatorname{Ric}_{g_{\mathrm{cap}}}
=(n-1)R_r^{-2}g_{\mathrm{cap}}$.
Writing $H_r=f_{\mathrm{cap}}'/f_{\mathrm{cap}}$, the sign conditions on
$q_r$ give
\begin{equation*}
 \frac{\Delta_{g_{\mathrm{cap}}}u_{\mathrm{cap}}}{u_{\mathrm{cap}}}
 =q_r'+q_r^2+(n-1)H_rq_r
 \le q_r^2\le\frac{C}{r^2}.
\end{equation*}
Because $R_r/r^{\beta_f}\to
c_\tau\lambda_\tau^{\beta_f}>0$ and $\beta_f>1$, the cap satisfies
\begin{equation*}
 \operatorname{Ric}_{g_{\mathrm{cap}}}
 -\alpha
 \frac{\Delta_{g_{\mathrm{cap}}}u_{\mathrm{cap}}}{u_{\mathrm{cap}}}
 g_{\mathrm{cap}}
 >(n-1)\kappa g_{\mathrm{cap}}
\end{equation*}
for all sufficiently small $r$.

Attach identical caps at the two ends of the truncated body
$[r,\pi/\lambda_\tau-r]$ and reflect the data across the midpoint.
The resulting warped data are $C^1$ at the two interfaces and satisfy the
strict tensor inequality on both sides.
Lemma~\ref{lem:quantitative-smoothing} produces smooth data
$(g_{\tau,r},u_{\tau,r})$ on $\mathbb S^n$ satisfying
\eqref{eq:sharpness-bound}.
The smoothing changes only the warping functions, not the radial coordinate,
so the distance between the two poles is
$d_{g_{\tau,r}}(N,S)=\pi/\lambda_\tau-2r+2\ell_r$.
Indeed, every curve joining the poles has length at least the total variation
of the radial coordinate, and a radial meridian realizes this length.

Choose $\tau_j\downarrow0$ and then $r_j\downarrow0$ sufficiently fast for
the cap and smoothing constructions above.
Equation \eqref{eq:cap-matching} gives
$d_{g_{\tau_j,r_j}}(N,S)\longrightarrow
\pi/\sqrt{a_\alpha(n-1)\kappa}$.
Theorem~\ref{thm:positive-diameter} bounds every diameter by the right-hand
side of \eqref{eq:sharpness-limit}.
The pole distance therefore squeezes the diameters to this value, which is
the right-hand side of \eqref{eq:sharpness-limit}.
\end{proof}

The singular core \eqref{eq:singular-spindle} is the equality profile of
the radial inequality \eqref{eq:positive-riccati}.
Related singular profiles also occur in spectral band rigidity
models~\cite{ChaiSun2026}; the capping argument above shows that the
Shen--Ye constant is nevertheless optimal within the smooth closed class.
For $n>3$, the constructed sequence degenerates:
$\max u_j/\min u_j\to\infty$, and the curvature of the two caps is
unbounded.

\section{Polynomial growth for mixed radial balls}\label{sec:polynomial}

We now prove the flexible mixed radial growth theorem.
For an admissible triple, first set
\begin{align}
 C_\alpha&:=\frac1\alpha-\frac{n-2}{n-1},
 &b&:=\frac p\alpha+\frac{n-3}{n-1},\label{eq:C-b}\\
 q&:=\frac1{n-1}-\frac{b^2}{4C_\alpha}.
 \label{eq:q-coefficient}
\end{align}
The first admissibility inequality is equivalent to $q>0$.
The second one will appear precisely as the integrability condition for the
negative tail of a mixed radial ray.

Fix a regular radial ray and suppress its angular variable.
Notice that $C_\alpha>0$.
Along this ray, let $\rho(s):=\int_0^su^p(t)\,dt$ and
$\tau:=\log\rho$.
Then $\tau_s=u^p/\rho>0$ and
$d/d\tau=(1/\tau_s)d/ds$.
Normalize the two logarithmic derivatives by $X:=m/\tau_s$ and
$Y:=\xi/\tau_s$.

The normalized variables satisfy an autonomous Riccati inequality.

\begin{lemma}\label{lem:dimensionless}
Define
\begin{equation}\label{eq:W}
 W:=Y+\frac{b}{2C_\alpha}X.
\end{equation}
Then
\begin{equation}\label{eq:X-inequality}
 X_\tau
 \le X-qX^2-C_\alpha W^2
 -(n-1)\kappa\rho^2u^{-2p}
 \le X-qX^2-C_\alpha W^2.
\end{equation}
\end{lemma}

\begin{proof}
Differentiating $\tau_s=u^p/\rho$ gives
$(\tau_s)_s/\tau_s=(p/\alpha)\xi-\tau_s$.
Consequently,
\begin{equation}\label{eq:X-derivative}
 X_\tau
 =\frac{m_s}{\tau_s^2}
 -X\left(\frac p\alpha Y-1\right).
\end{equation}
Discard the nonpositive right-hand side of \eqref{eq:exact-riccati}, divide
by $\tau_s^2$, and use \eqref{eq:X-derivative} to obtain
\begin{equation*}
 X_\tau
 \le X-\frac1{n-1}X^2-bXY-C_\alpha Y^2
 -(n-1)\kappa\rho^2u^{-2p}.
\end{equation*}
The Cauchy--Schwarz inequality applied to the terms involving $Y$ proves
\eqref{eq:X-inequality}.
\end{proof}

For the target density, set
\begin{equation}\label{eq:delta-beta}
 \delta:=\frac{\sigma-p-\alpha}{\alpha},
 \qquad
 \beta:=1-\frac{\delta b}{2C_\alpha}.
\end{equation}
The density of $u^\sigma dV_g$ relative to $d\rho$ is
\begin{equation}\label{eq:target-density}
 \mathcal J:=u^{\sigma-p-\alpha}j,
 \qquad
 u^\sigma J\,ds=\mathcal J\,d\rho.
\end{equation}
Direct differentiation and substitution of \eqref{eq:W} give
\begin{equation}\label{eq:target-derivative}
 (\log\mathcal J)_\tau=X+\delta Y=\beta X+\delta W.
\end{equation}

The correction coefficient and the resulting growth exponent are
\begin{align}
 \Gamma&:=\sqrt{\beta^2+\frac{q}{C_\alpha}\delta^2}>0,
 \label{eq:Gamma}\\
 K(\alpha,p,\sigma)&:=1+\frac{\beta+\Gamma}{2q}.
 \label{eq:K}
\end{align}
Before estimating the density, we record the lower bound for the exponent.

\begin{lemma}\label{lem:K-lower}
For every admissible triple, $K\ge n$.
Equality holds if and only if $\sigma=np+(n-2)\alpha$.
\end{lemma}

\begin{proof}
By \eqref{eq:K}, the desired inequality is equivalent to
\begin{equation}\label{eq:K-goal}
 \beta+\Gamma\ge2(n-1)q.
\end{equation}
If $2(n-1)q-\beta\le0$, this is immediate and strict.
Otherwise, direct substitution from \eqref{eq:C-b},
\eqref{eq:q-coefficient}, \eqref{eq:delta-beta}, and \eqref{eq:Gamma} gives
\begin{equation*}
 \Gamma^2-\bigl(2(n-1)q-\beta\bigr)^2
 =\frac{q}{C_\alpha}
 \bigl(\delta-(n-1)b\bigr)^2\ge0,
\end{equation*}
which proves \eqref{eq:K-goal}.
Equality holds precisely when $\delta=(n-1)b$; in that case
$2(n-1)q-\beta=1$, so equality indeed occurs.
Expanding $\delta=(n-1)b$ gives
$\sigma=np+(n-2)\alpha$.
\end{proof}

The next estimate controls the phase in which the normalized radial mean
curvature is bounded below.

\begin{lemma}\label{lem:controlled-phase}
Along every regular radial ray,
\begin{equation}\label{eq:corrected-monotone}
 \frac d{d\tau}
 \left(\log\mathcal J+\Gamma X-(K-1)\tau\right)\le0.
\end{equation}
There are constants $\rho_*>0$ and $C>0$, independent of the radial
direction, such that on every interval with $\rho\ge\rho_*$ and $X\ge-1$,
\begin{equation}\label{eq:J-controlled}
 \mathcal J(\rho)\le C\rho^{K-1}.
\end{equation}
\end{lemma}

\begin{proof}
Equations \eqref{eq:X-inequality} and
\eqref{eq:target-derivative} give
\begin{align*}
 \frac d{d\tau}(\log\mathcal J+\Gamma X)
 \le{}&(\beta+\Gamma)X+\delta W
 -\Gamma qX^2-\Gamma C_\alpha W^2.
\end{align*}
The maximum of the right-hand side over $(X,W)\in\mathbb R^2$ is
$(\beta+\Gamma)^2/(4\Gamma q)+\delta^2/(4\Gamma C_\alpha)$.
By \eqref{eq:Gamma} and \eqref{eq:K}, this maximum equals $K-1$.
This proves \eqref{eq:corrected-monotone}.

It remains to choose a uniform starting point.
As $s\downarrow0$, uniformly in the angular variable,
\begin{equation}\label{eq:pole-polynomial}
 \begin{aligned}
 j(s,\theta)&=u^\alpha(o)s^{n-1}(1+O(s)),\\
 \rho(s,\theta)&=u^p(o)s(1+O(s)),\\
 X(s,\theta)&=n-1+O(s).
 \end{aligned}
\end{equation}
The $h_\alpha$-injectivity radius at $o$ is positive, and the initial sphere
is compact.
Thus one may choose a single $\rho_*>0$ such that all rays are regular up to
$\rho_*$, $X(\rho_*,\theta)>-1$, and
$\log\mathcal J+\Gamma X-(K-1)\log\rho$ is uniformly bounded above there.
Integrating \eqref{eq:corrected-monotone} from this positive radius and using
$X\ge-1$ proves \eqref{eq:J-controlled}.
Starting at a positive radius is necessary when $K>n$, because the corrected
expression need not have a finite limit at the pole.
\end{proof}

Suppose now that $X$ first reaches $-1$ at $\tau=\tau_1$.
At any point with $X=-1$, the right-hand side of
\eqref{eq:X-inequality} is at most $-1-q<0$.
Therefore $X$ cannot cross back above $-1$.
On the remaining part of the ray, set $Z:=-X>1$.
Then
\begin{equation}\label{eq:Z-growth}
 Z_\tau\ge Z+qZ^2+C_\alpha W^2.
\end{equation}
The volume density relative to $d\tau$ is $P:=\rho\mathcal J$, so that
$\mathcal J\,d\rho=P\,d\tau$, and
\begin{equation}\label{eq:log-P}
 (\log P)_\tau=1-\beta Z+\delta W.
\end{equation}

The second admissibility condition is exactly what makes the negative tail
integrable.

\begin{lemma}\label{lem:tail}
The inequality
\begin{equation}\label{eq:tail-abstract}
 \delta^2<4C_\alpha(\beta+q)
\end{equation}
is equivalent to \eqref{eq:admissible-sigma}.
Under this condition, there is a number $c\in(0,1)$, depending only on the
fixed parameters, such that for every $Z\ge1$ and $W\in\mathbb R$,
\begin{equation}\label{eq:tail-comparison}
 1-\beta Z+\delta W
 \le c\left(1+qZ+C_\alpha\frac{W^2}{Z}\right).
\end{equation}
\end{lemma}

\begin{proof}
Direct substitution gives
\begin{equation*}
 4C_\alpha(\beta+q)-\delta^2
 =\frac{
 4n\alpha\bigl((n-1)-(n-2)\alpha\bigr)
 -\bigl((n-1)\sigma-2\alpha\bigr)^2
 }{(n-1)^2\alpha^2}.
\end{equation*}
This proves the equivalence.

At $c=1$, the quadratic polynomial
$\beta+q-\delta t+C_\alpha t^2$ has positive minimum by
\eqref{eq:tail-abstract}.
For $c<1$ sufficiently close to one, continuity of the minimum gives
$\beta+cq-\delta t+cC_\alpha t^2\ge1-c$ for $t\in\mathbb R$.
Set $t=W/Z$, multiply by $Z\ge1$, and rearrange to obtain
\eqref{eq:tail-comparison}.
\end{proof}

We can now control the entire contribution of a single ray.

\begin{proposition}\label{prop:raywise}
There is a constant $C>0$, independent of the radial direction, such that on
every finite regular radial segment,
\begin{equation}\label{eq:raywise-polynomial}
 \int_0^{\rho(s)}\mathcal J(t)\,dt
 \le C\left(1+\rho^K(s)\right).
\end{equation}
\end{proposition}

\begin{proof}
The part $0\le\rho\le\rho_*$ is uniformly bounded by
\eqref{eq:pole-polynomial}.
As long as $X\ge-1$, Lemma~\ref{lem:controlled-phase} gives
\begin{equation*}
 \int_{\rho_*}^{\rho}\mathcal J(t)\,dt
 \le C\int_{\rho_*}^{\rho}t^{K-1}\,dt
 \le C(1+\rho^K).
\end{equation*}

If $X$ never reaches $-1$, this proves the proposition.
Otherwise, let $\rho_1=e^{\tau_1}$ be the first crossing radius.
By \eqref{eq:Z-growth}, \eqref{eq:log-P}, and Lemma~\ref{lem:tail}, we have
$(\log P)_\tau\le c(\log Z)_\tau$.
Since $Z(\tau_1)=1$, it follows that
$P(\tau)\le P(\tau_1)Z^c(\tau)$.
Also, \eqref{eq:Z-growth} implies $d\tau\le dZ/(qZ^2)$.
Hence the entire negative tail satisfies
\begin{equation}\label{eq:tail-volume}
 \int_{\tau_1}^{\tau}P(t)\,dt
 \le\frac{P(\tau_1)}{q}
 \int_1^\infty Z^{c-2}\,dZ
 =\frac{P(\tau_1)}{q(1-c)}.
\end{equation}
At the crossing, \eqref{eq:J-controlled} yields
$P(\tau_1)=\rho_1\mathcal J(\rho_1)\le C\rho_1^K$.
Since $\rho$ is increasing and $K>0$, combining the controlled part with
\eqref{eq:tail-volume} proves \eqref{eq:raywise-polynomial}.
\end{proof}

\begin{proof}[Proof of Theorem~\ref{thm:polynomial}]
It remains to globalize Proposition~\ref{prop:raywise}.

For $S_0>0$, truncate every ray at $g$-arclength $S_0$, at its cut time,
or when $\rho_p=R$, whichever occurs first, and denote the resulting radial
set by $\Omega_{R,S_0}^{\alpha,p}(o)$.
The polar map is one-to-one before the cut time.
Using \eqref{eq:target-density} and Proposition~\ref{prop:raywise} on each
truncated ray gives
$\int_{\Omega_{R,S_0}^{\alpha,p}(o)}u^\sigma\,dV_g\le C(1+R^K)$.
The angular factor $\omega_{n-1}$ has been absorbed into $C$.
Letting $S_0\to\infty$ and using monotone convergence proves
\eqref{eq:polynomial-growth}.
The limiting argument also covers rays on which $\rho_p<R$ along their
entire infinite continuation.
\end{proof}

If $h_\alpha$ is complete, the theorem contains the conformal metric balls
as a direct specialization.
Whenever $(\alpha,\alpha,n\alpha)$ is admissible,
\begin{equation*}
 \operatorname{Vol}_{h_\alpha}(B_R^{h_\alpha}(o))
 =\int_{B_R^{h_\alpha}(o)}u^{n\alpha}\,dV_g
 \le C\left(1+R^{K(\alpha,\alpha,n\alpha)}\right).
\end{equation*}
Other choices of $p$ produce genuinely mixed radial balls rather than metric
balls, but the proof is unchanged.

\section{The spectral Bonnet--Myers and volume comparison theorem}\label{sec:AX}

Theorem~\ref{thm:positive-diameter} gives a diameter estimate independent of
$u$ when $0<\alpha<4/(n-1)$.
We prove Corollary~\ref{cor:intro-AX} on the wider model-comparison range,
where the diameter estimate retains the oscillation of $u$.

We first record the oscillation-dependent diameter estimate on this wider
parameter range.

\begin{lemma}\label{lem:diameter}
Suppose $g$ is complete, $0<\alpha\le(n-1)/(n-2)$, $\kappa>0$, and
\eqref{eq:spectral} holds.
If $0<\inf_M u\le\sup_M u<\infty$, then
\begin{equation}\label{eq:diameter-bound}
 \operatorname{diam}_g(M)
 \le\frac{\pi}{\sqrt\kappa}
 \left(\frac{\sup_M u}{\inf_M u}\right)^\eta.
\end{equation}
In particular, $M$ is compact.
\end{lemma}

\begin{proof}
The lower bound for $u$ and the completeness of $g$ imply that
$h_\alpha=u^{2\alpha}g$ is complete.
Let $\gamma$ be any minimizing $h_\alpha$-geodesic from a fixed point and
parametrize it by $g$-arclength.
Use the variable $z$ from \eqref{eq:z} in the model equation
\eqref{eq:Y-equation}.
Since $U\ge\inf_M u$, the first positive zero $z_0$ of $Y$ satisfies
$z_0\le\pi/(\sqrt\kappa\,(\inf_M u)^\eta)$.
Lemma~\ref{lem:root-comparison} shows that the cut time occurs no later than the
corresponding $g$-arclength time.
On the other hand,
$z(s)=\int_0^sU^{-\eta}(t)\,dt\ge s/(\sup_M u)^\eta$.
Consequently every minimizing $h_\alpha$-geodesic has $g$-length at
most the right-hand side of \eqref{eq:diameter-bound}.
For any two points, their $g$-distance is no larger than the $g$-length of
such a curve, proving the diameter estimate.
Since $g$ is complete and $M$ has finite diameter, the Hopf--Rinow theorem
implies that $M$ is compact.
\end{proof}

\begin{proof}[Proof of Corollary~\ref{cor:intro-AX}]
Lift $g$ and $u$ to the universal cover.
The lifted metric is complete, the lifted function has a positive minimum
and finite maximum, and \eqref{eq:spectral} remains valid.
Lemma~\ref{lem:diameter} proves \eqref{eq:AX-diameter} and shows that
$\widetilde M$ is compact.
A compact universal cover has only finitely many deck transformations, so
$\pi_1(M)$ is finite.

Choose a lift $o$ of a minimum point of $u$.
The metrics $g$ and $h_\alpha$ are uniformly equivalent on
$\widetilde M$.
Consequently, for sufficiently large $R$, the radial ball
$\Omega_R^\alpha(o)$ covers $\widetilde M$ up to the cut locus.
Letting $R\to\infty$ in \eqref{eq:minimum-growth} gives
\begin{align*}
 \operatorname{Vol}_{\widetilde g}(\widetilde M)
 &\le
 \omega_{n-1}\int_0^{\pi/\sqrt\kappa}s_\kappa^{\,n-1}(t)\,dt
 =\kappa^{-n/2}\operatorname{Vol}(\mathbb S^n).
\end{align*}
This proves \eqref{eq:AX-volume}.

If equality holds in \eqref{eq:AX-volume}, inspection of the proof of
\eqref{eq:minimum-growth} shows that $\widetilde u$ is constant.
Therefore $u$ is constant.
The curvature assumption reduces to
$\operatorname{Ric}_g\ge(n-1)\kappa g$.
Rigidity in the classical Bishop volume comparison then identifies
$\widetilde M$ with the round sphere of radius $\kappa^{-1/2}$.
\end{proof}

\section{The stable Bernstein theorem in
\texorpdfstring{$\R^4$}{R4}}\label{sec:bernstein}

\begin{proof}[Proof of Corollary~\ref{cor:bernstein}]
By stability, there is a smooth positive function $u$ satisfying
$\Delta_gu+|A|^2u=0$; see
Fischer-Colbrie~\cite[Proposition~1]{FischerColbrie1985} and Catino,
Mastrolia, and Roncoroni~\cite[Section~2.1]{CatinoMastroliaRoncoroni2024}.
Set $h:=u^{4/3}g$ and normalize $u(o)=1$ at a fixed base point.
The Euclidean Gauss equation and the algebraic estimate for a trace-free
symmetric tensor give
$\operatorname{Ric}_g-(2/3)(\Delta_gu/u)g=-A^2+(2/3)|A|^2g\ge0$; see also
Cabr{\'e}, Catino, Mari, Mastrolia, and
Roncoroni~\cite[(2.2)]{CabreCatinoMariMastroliaRoncoroni2026}.
Moreover, Catino, Mastrolia, and
Roncoroni~\cite[Lemma~2.3]{CatinoMastroliaRoncoroni2024} prove that the
metric $h=u^{4/3}g$ is complete.

The parameters $n=3$, $\alpha=p=2/3$, and $\sigma=2$ satisfy the
admissibility conditions in Theorem~\ref{thm:polynomial}, and
formula~\eqref{eq:K} gives $K=2+\sqrt2$.
Since the corresponding mixed radial regions are the $h$-balls and
$u^2dV_g=dV_h$, the theorem gives
\begin{equation}\label{eq:Bernstein-volume}
 \operatorname{Vol}_h(B_R^h(o))
 \le C_o(1+R^{2+\sqrt2})
 \qquad(R>0).
\end{equation}
Corollary~3.7 and equation~(3.9) of
Cabr{\'e}, Catino, Mari, Mastrolia, and
Roncoroni~\cite{CabreCatinoMariMastroliaRoncoroni2026} with
$n=3$, $\tau=1/3$, $\delta=2$, and $t=1/6$ give
\begin{equation*}
 \frac13\int_M|A|^4u^{-2/3}\psi^2\,dV_g
 \le\int_M|A|^2u^{-2/3}|\nabla^g\psi|_g^2\,dV_g.
\end{equation*}
Taking $\psi=\zeta^2$, applying H\"older's inequality, and using
$dV_h=u^2dV_g$ and
$|\nabla^h\zeta|_h=u^{-2/3}|\nabla^g\zeta|_g$ yield
\begin{equation}\label{eq:Bernstein-L4}
 \int_M\left(\frac{|A|}{u^{2/3}}\right)^4\zeta^4\,dV_h
 \le144\int_M|\nabla^h\zeta|_h^4\,dV_h
\end{equation}
for every compactly supported Lipschitz function $\zeta$.

By completeness, choose an $h$-distance cutoff $\zeta_R$ that equals one on
$B_R^h(o)$, vanishes outside $B_{2R}^h(o)$, and satisfies
$|\nabla^h\zeta_R|_h\le2/R$ almost everywhere.
Combining \eqref{eq:Bernstein-volume} and \eqref{eq:Bernstein-L4}, we obtain
\begin{align*}
 \int_{B_R^h(o)}
 \left(\frac{|A|}{u^{2/3}}\right)^4\,dV_h
 &\le\frac{C}{R^4}\operatorname{Vol}_h(B_{2R}^h(o))
 \notag\\
 &\le C\left(R^{-4}+R^{\sqrt2-2}\right)
 \longrightarrow0.
\end{align*}
Hence $A\equiv0$.
The immersion is therefore a local isometry into an affine three-plane, and
completeness implies that its image is the entire plane.
\end{proof}

\bibliographystyle{amsalpha}
\bibliography{references_pseudoball_applications}

\end{document}